%% file: unsafe_arxiv.tex
\documentclass[dvipsnames]{article}
\usepackage[utf8]{inputenc}
\usepackage[margin=1in]{geometry}
\usepackage{amsmath} 
\usepackage{todonotes}
\input{macro/jared_macro.tex}

\input{macro/prob_macro}

\usepackage{amsthm}
\newtheorem{thm}{Theorem}[section]
\newtheorem{lem}[thm]{Lemma}
\newtheorem{prop}[thm]{Proposition}
\newtheorem{cor}{Corollary}
\newtheorem{defn}{Definition}[section]

\newtheorem{rem}{Remark}

\renewcommand\footnotemark{}

\newlength{\exfiglength}
\title{\textbf{Unsafe Probabilities and Risk Contours for Stochastic Processes using Convex Optimization
}}

\author{Jared Miller$^1$, Matteo Tacchi$^2$, Didier Henrion$^3$, Mario Sznaier$^4$
\thanks{$^1$ J. Miller is with the Automatic Control Laboratory (IfA), Department of Information Technology and Electrical Engineering (D-ITET), ETH Z\"{u}rich, Physikstrasse 3, 8092, Z\"{u}rich, Switzerland (e-mail: jarmiller@control.ee.ethz.ch).}
\thanks{$^2$ M. Tacchi is with Univ. Grenoble Alpes, CNRS, Grenoble INP (Institute of Engineering Univ. Grenoble Alpes), GIPSA-lab, 38000 Grenoble, France. (e-mail: matteo.tacchi@gipsa-lab.fr)}
\thanks{$^3$ D. Henrion is with the Polynomial Optimization group of LAAS-CNRS, Universit\'e de Toulouse, CNRS, Toulouse, France; and the 
Faculty of Electrical Engineering, Czech Technical University in Prague, Czech Republic. (henrion@laas.fr)}
\thanks{$^4$ M. Sznaier is with the Robust Systems Lab,  ECE Department, Northeastern University, Boston, MA 02115. (e-mail: msznaier@coe.neu.edu).}
\thanks{J. Miller and M. Sznaier were partially supported by NSF grants  CNS--1646121, ECCS--1808381 and CNS--2038493, AFOSR grant FA9550-19-1-0005, and ONR grant N00014-21-1-2431.  J. Miller was partially supported by the Swiss National Science Foundation under NCCR Automation, grant agreement 51NF40\_180545.
}}

\begin{document}

\maketitle


\input{sections/abstract}

\input{sections/introduction}

\input{sections/summary}
\input{sections/preliminaries}
\input{sections/methods}
\input{sections/unsafe_sos}
\input{sections/examples}
\input{sections/Conclusion}

\input{sections/acknowledgements}

J. Miller and M. Sznaier were partially supported by NSF grants  CNS--2038493 and CMMI--2208182, AFOSR grant FA9550-19-1-0005,  ONR grant N00014-21-1-2431, and DHS grant 22STESE0001-02-00.
J. Miller was partially supported by the Swiss National Science Foundation under NCCR Automation, grant agreement 51NF40\_180545.
J. Miller was in part funded by Deutsche Forschungsgemeinschaft (DFG, German Research Foundation) under Germany's Excellence Strategy - EXC 2075 – 390740016. We acknowledge the support by the Stuttgart Center for Simulation Science (SimTech).



\appendix
\input{appendix/app_no_relaxation_measure}
\input{appendix/app_duality}
\input{appendix/app_risk_upper_bound}
\input{appendix/app_risk_convergence}
\input{appendix/app_converge_sos}
\bibliographystyle{IEEEtran}        
\bibliography{references}  

\end{document}

%% file: macro/jared_macro.tex
\usepackage[nolist]{acronym}
\usepackage{amsmath}
\usepackage{amssymb}
\usepackage{mathtools}
\usepackage{subcaption}
\usepackage{hyperref}
\usepackage{xcolor}
\hypersetup{colorlinks=true, linkcolor=black, citecolor=OliveGreen}
\usepackage{algorithm}
\usepackage{algpseudocode} 
\usepackage{mathrsfs}
\usepackage{cite}
\usepackage{adjustbox}
\usepackage{float}
\usepackage{wrapfig}

\newcommand{\rw}[1]{{\color{black} #1}}
\newcommand{\ra}[1]{{\color{black} #1}}

\newcommand{\E}{\mathbb{E}} 
\newcommand{\R}{\mathbb{R}} 
\newcommand{\K}{\mathbb{K}} 
 
\newcommand{\N}{\mathbb{N}}

\newcommand{\Lie}{\mathcal{L}}

\newcommand{\A}{\mathcal{A}}

\newcommand{\psd}{\mathbb{S}}

\newcommand{\cs}{\mathcal{C}}

\DeclarePairedDelimiter{\norm}{\lVert}{\rVert}

\DeclarePairedDelimiterX{\inp}[2]{\langle}{\rangle}{#1, #2}

\DeclarePairedDelimiter{\Mp}{\mathcal{M}_+(}{)}



%% file: macro/prob_macro.tex
\newcommand{\prb}{\textrm{Prob}}

\newcommand{\bbmu}{\boldsymbol{\mu}}
\newcommand{\bell}{\boldsymbol{\ell}}

%% file: sections/abstract.tex
\begin{abstract}
\label{sec:abstract}
\rw{When evaluating safety specifications for trajectories of a dynamical system, it is vital to be able to bound the worst-case probability of unsafety (constraint violation).}
\rw{Certifications of stochastic safety and worst-case probabilities of unsafety can be expressed as infinite-dimensional linear programs (e.g. stochastic barrier functions, occupation measure problems). }
\rw{This paper proves that the infinite-dimensional linear programs and their finite-dimensional Moment-Sum-of-Squares truncations} are nonconservative (to the true probability of unsafety) under compactness and regularity conditions in \rw{stochastic} dynamics. 
Unsafe-probability \rw{estimates} and risk contours are generated for example stochastic processes. 




\end{abstract}

%% file: sections/introduction.tex
\section{Introduction}
\label{sec:introduction}


\rw{A trajectory of a dynamical system is considered safe if it never contacts a hazardous set $X_u$ over the course of its execution starting from an initial set $X_0$. Instances of safety specifications for a deterministic system could include an aircraft not crashing into the ground, an autonomous vehicle not hitting a pedestrian, or the temperature in a room not violating a comfort constraint. 
In a stochastic setting, such declarations of safety can be formulated as a safety specification holding \ra{almost surely} \cite{CLARK2021stochcbf}. However, \ra{almost sure} guarantees of safety can be either overly conservative or impossible to meet (in the case of ergodic processes such as non-vanishing \ra{G}aussian noise). Safety specifications can therefore be certified according to $(1-p)$-safety, in which the \ra{hazardous} set $X_u$ is contacted with worst-case probability $p$ starting from $X_0$.
Prior work on certifying $(1-p)$ safety for fixed $p$ using (control) barrier functions include \cite{prajna2004stochastic, salamati2021data, SANTOYO2021109439}. These methods provide sufficient conditions for safety by exploiting a supermartingale probability inequality (Lemma 1 in \cite{Kushner67}), but they do not provide conditions for existence, nonconservatism, nor convergence.
Other stochastic verification methods include reachability \cite{digailova2004reachability, prandini2006stochastic}, reach-avoid analysis \cite{MOHAJERINESFAHANI201643, bansal2017hamilton}, ruin \cite{asmussen2010ruin}, and compositional abstraction \cite{lavaei2020compositional, nejati2022compositional}.

The focus of the present paper is to exactly compute the worst-case probability of unsafety (reaching $X_u$), starting from $X_0$, and to provide non-conservative conditions under which this worst-case unsafe probability can be computed without conservatism using convex optimization methods. The worst-case unsafety problem is generally a nonconvex but finite-dimensional optimization problem. Using established methods of occupation-measure-based lifting from optimal control \cite{rubio1975generalized, lewis1980relaxation}, the nonconvex unsafety problem will be lifted into an infinite-dimensional \ac{LP} in occupation measures.
The main contribution of this work is to present \ac{LP} formulations and conditions for which:
\begin{enumerate}
    \item The convex \ac{LP} objective will equal the worst-case probability of unsafety.
    \item A risk contour (continuous function with arguments of initial time and initial condition) derived from  \iac{LP} will converge from above in an $L_1$ sense to the true worst-case probability achieved.
\end{enumerate}}

The \rw{developed infinite-dimensional \acp{LP} in this paper } must be truncated into finite-dimensional programs in order to obtain numerical solutions. Algorithms to perform this truncation for generic stochastic processes include gridding \cite{cho2002linear}, radial basis function selection \cite{kariotoglou2017linear},  random sampling \cite{huynh2012incremental}, \ra{and neural network parameterization with counterexample-guided verification \cite{abate2018counterexample, abate2021fossil}}.
In the specific case where the stochastic process and all problem data have a rational representation, the moment-\ac{SOS} hierarchy \cite{lasserre2009moments} can be employed to truncate the \ac{LP} into a hierarchy of \acp{SDP} in increasing size, parameterized by the polynomial degree. Application of the moment-\ac{SOS} hierarchy for stochastic analysis and control includes regional verification \cite{steinhardt2012finite}, reach-avoid estimation \cite{xue2023reach},
exit-time estimation \cite{henrion2021moment}, and (conditional) value-at-risk upper-bounding \cite{miller2023chancepeak}.

 The two most relevant prior works to our paper are \cite{sloth2015safety} and \cite{jasour2019risk}. The work in \cite{sloth2015safety} presents infinite-dimensional \acp{LP} to generate superlevel-based outer approximations of probability-\rw{$(1-p)$}-safe sets of a \ac{SDE} for a fixed initial distribution $\mu_0$ \rw{and probability level $p$}. 
Our problem evaluates safety of a specific initial set $X_0$ with respect to a free initial distribution, performs risk contour analysis, and extends the framework towards more general stochastic processes. The arguments that we use in posing this \ac{LP} are similar \ra{to} steps taken in Lemma 2 of the previously presented work \cite{sloth2015safety}. The work in \cite{jasour2019risk} synthesizes risk contours for optimization problems with distributional uncertainty (with an emphasis on obstacles and path planning), and our paper extends this work by analyzing safety of stochastic dynamical systems.




%% file: sections/summary.tex
This paper has the following structure: 
Section \ref{sec:preliminaries} introduces preliminaries such as notation, stochastic processes, and \ac{SOS} proofs of nonnegativity. Section \ref{sec:unsafe_lp} poses infinite-dimensional convex \acp{LP} to compute the worst-case probability of unsafety of a stochastic process w.r.t. an initial set $X_0$ \rw{or an initial distribution $\mu_0$.}
Section \ref{sec:unsafe_sos} truncates the unsafety \acp{LP} into a converging sequence of  finite-dimensional \acp{SDP} using the moment-\ac{SOS} hierarchy.
Section \ref{sec:examples} demonstrates the computation of unsafe probabilities and risk contours on example stochastic process with polynomial dynamics.
Section \ref{sec:conclusion} concludes the paper. 


%% file: sections/preliminaries.tex
\section{Preliminaries}
\label{sec:preliminaries}

\input{sections/acronym}

\subsection{Notation}

The subset of natural numbers between $a$ and $b$ is $a..b \subset \N$.
The minimum of two quantities will be denoted as $a \wedge b=\min(a,b)$. The set of polynomials with real-valued coefficients in an indeterminate $x \in \R^n$ is $\R[x]$. Every polynomial $c \in \R[x]$ may be uniquely described by a finite sum over multi-indices $\alpha \in \N^n$ by $c(x) = \sum_{\alpha} c_\alpha x^\alpha$. The degree of a polynomial $\deg{c}$ is equal to the largest exponent-sum $\sum_{i=1}^n \alpha_i$ such that $c_\alpha \neq 0$. The set of polynomials in $x$ with degree at most $2k$ is $\R[x]_{\leq 2k}$.

\subsection{Analysis and Measure Theory}
The set of continuous functions over a set $S$ is $C(S)$. Its subcone of nonnegative continuous functions is $C_+(S)$. Given a product set $S \times H$, the set of functions that are once-continuously differentiable in the first variable $s$ and are twice-continuously differentiable in the second variable $h$ is $C^{1, 2}(S \times H)$. The set of signed Borel measures supported in $S$ is $\mathcal{M}(S)$, and its subset of nonnegative Borel measures supported in $S$ is $\Mp{S}$. 
The support of a measure $\mu \in \Mp{\R^n}$ is the smallest closed set of points $s$ such that every open neighborhood $N_\epsilon(s)$ has positive measure $\mu(N_\epsilon(s)) > 0$.
The sets $C_+(S)$ and $\Mp{S}$ are in topological duality when $S$ is compact, and they admit a duality pairing $\inp{\cdot}{\cdot}$ by Lebesgue integration: $\forall f \in C_+(S), \ \mu \in \Mp{S}: \inp{f}{\mu} = \int_S f(s) d \mu(s)$. This duality pairing will be extended into a bilinear pairing between $C(S)$ and $\mathcal{M}_+(S)$. The $\mu$-measure of a set $A \subseteq S$ may also be expressed as the pairing of $\mu$ with the indicator function $I_A: \ \mu(A) = \inp{I_A}{\mu} = \inf \{\inp{w}{\mu} \; | \; w \in C(S), w \geq I_A\}$. The mass of $\mu$ is $\mu(S) = \inp{1}{\mu}$, and $\mu$ is a probability distribution if this mass is 1. A vital example of a probability distribution is the Dirac delta $\delta_{s'}$ supported only at $s'$, obeying the point-evaluation pairing rule $\forall f \in C(S): \ \inp{f}{\delta_{s'}} = f(s').$ 
\rw{When $\mu$ is a probability distribution, the probability of the event $x \in A$ is $\text{Prob}_\mu[x \in A] = \mu(A)$.}
The product measure between $\mu \in \Mp{S_1}$ and $\nu \in \Mp{S_2}$ is the unique measure $\mu \otimes \nu$ satisfying $\forall A_1 \times A_2 \subseteq S_1 \times S_2: \ \mu \otimes \nu(A_1 \times A_2) = \mu(A_1)\nu(A_2).$
\vspace{-0.2cm}

\subsection{Stochastic Processes}
\vspace{-0.2cm}
A stochastic process is a time-indexed set of random variables $\{\rw{\mu_t}\}$ that are related together through system dynamics \cite{stroock1997multidimensional} (pushforward of an initial distribution along flow maps). Properties of the stochastic process will be analyzed in terms of test functions in an appropriate set $\cs$ (such as $\cs = C([\rw{t_0}, T] \times X)$.
The expectation of a test function $v(s, x)$ at the time $t$ according to the distribution $\rw{\mu_t}$ is $\E[v(t,x) \mid \rw{\mu_t}]$. Letting ${\Delta t} > 0$ be a time step, the \textbf{generator} $\Lie_{\rw{\Delta t}}$ of a stochastic process satisfies (for all appropriate test functions $v \in \textrm{dom}{\Lie_{\rw{\Delta t}}}$)
\begin{equation*}
    \Lie_{{\Delta t}} v(t, x) = \lim_{{{\Delta t}}' \rightarrow {{\Delta t}}} \frac{\E[v(t+{{\Delta t}}',x) \mid \rw{\mu}_{t+{{\Delta t}}'}] - v(t,x) }{{{\Delta t}}'}. 
\end{equation*}
We will express the domain of $\Lie_{\rw{\Delta t}}$ as $\cs  = \textrm{dom}{\Lie_{\rw{\Delta t}}}$, such that $\cs$ is a subset of the preimage of  continuous functions under $\Lie_{\rw{\Delta t}}$.
The generator for a discrete-time stochastic process is $\Lie_{{{\Delta t}}}$ with ${{\Delta t}} > 0$, which is defined w.r.t. the test function class $\cs = C([\rw{t_0}, T] \times X)$. For a discrete-time law following $x_{t+{{\Delta t}}} = f(t, \rw{x_t}, \lambda_t)$ in which the time-varying parameter $\lambda_t \in \Lambda$ has a probability distribution of $\xi(\lambda_t)$, the associated generator satisfies
\begin{align*}
    \Lie_{{\Delta t}} v(t, x) = \rw{\frac{1}{\Delta t}} \int_{\rw{\lambda \in }\Lambda} v\left(t+{{\Delta t}}, f(t, x, \lambda)\right) -v(t, x) \ d \xi(\lambda).
\end{align*}
The generator for a continuous-time (Feller) stochastic process is $\Lie_{{{\Delta t}}=0}$ for the class $\cs = C^{1, 2}([\rw{t_0}, T] \times X)$ \cite{rogers2000diffusions}. 
The dynamical behavior of an It\^{o} \ac{SDE} \rw{with a drift $f$ and a diffusion $g$} has a description
\begin{align}
    & dx = f(t, x) dt + g(t, x) dW, \label{eq:sde}\\
    \intertext{in which $W$ is the Wiener process. The generator of the \ac{SDE} in \eqref{eq:sde} is}
        \Lie_0 v(t, x) =& \partial_t v(t, x) + f(t, x) \cdot \nabla_x v(t,x) \nonumber\\
        &+ g(t,x)^T\nabla^2_{xx}v(t,x) g(t,x)/2. \label{eq:lie}
\end{align}
A sequence of random variables $\{Y_t\}$ is a \textit{martingale} if $\E[Y_{t+\delta t} \mid \{Y_{t'}\}_{t' \leq t}] = Y_t$ \cite{stroock1997multidimensional} (in the sense of generalized conditional expectation).
The generator $\Lie_{\Delta t}$ solves a martingale problem for all possible time steps $s > 0$ and test functions $v \in \cs$ \cite{williams1991probability}. 
The martingale problem is
\begin{align}
    0 = \E[v(t+s, x) \mid \rw{\mu_{t + s}} ] &- \E[v(t, x) \mid \rw{\mu_{t}} ]\label{eq:martingale}\\
    &\textstyle \quad - \sum_{s'=t}^{t+s} \E[\Lie_{{\Delta t}} v(t, x) \mid \rw{\mu_{s'}}],  \nonumber
\end{align}
\ra{for which the continuous-time problem involves $\Delta t = 0$.}
\rw{The expression in \eqref{eq:martingale} is known as Dynkin's formula when $\Lie_0$ describes an \ac{SDE} (\eqref{eq:sde}), or Liouville's formula when $\Lie_0$ models an \ac{ODE} (\ac{SDE} with $g=0$ everywhere).}
\rw{We will denote the choice of the discrete-time generator $\Lie_{\rw{\Delta t}}$ or continuous-time generator $\Lie_0$ as the symbol $\Lie$}, and will refer to the generator $\Lie$ interchangeably with its stochastic process.


\subsection{Occupation Measures}
The martingale problem can be described using the theory of occupation measures. Given an initial condition (point) $x_0 \in X_0 \subseteq X$ and a time $t \in [\rw{t_0}, T]$, we refer to $x(t \mid x_0)$ as the trajectory of $\{\rw{\mu_t}\}$ with generator $\Lie$ starting at the initial distribution $\delta_{x = x_0}$.
The \textbf{occupation measure} $\mu \in \Mp{[\rw{t_0}, T] \times X}$ and \textbf{terminal measure} $\mu_\tau \in \Mp{[\rw{t_0}, T] \times X}$ of the stochastic process $\{\rw{\mu_t}\}$ from a starting time $t_0$ up to a stopping time $t^*$ satisfies $\forall A \subseteq [t_\rw{t_0}, T], B \subseteq X$:
\begin{subequations}
\label{eq:measures_prob}
\begin{align}    
    \mu(A \times B) &= \textstyle \int_{X_0} \int_{t=t_0}^{t^*}  I_{A \times B}\left(t, x(t \mid x_0)\right) dt \, d\mu_0(x_0) \nonumber \\        
    \mu_\tau(A \times B) &= \textstyle \int_{X_0} I_{A \times B}\left(t^*, x(t^* \mid x_0)\right) \, d\mu_0(x_0).
\end{align}
\end{subequations}
The measures in \eqref{eq:measures_prob} may also be defined with respect to a probability distribution of stopping times $t^* \in \Mp{[\rw{t_0}, T]}$.
A tuple of measures $(\mu_0, \rw{ \mu_\tau, \mu}) \in \Mp{X_0} \times \Mp{[\rw{t_0}, T] \times X}^2$ obeys the martingale relations \eqref{eq:martingale}  of the stochastic process with generator $\Lie$ 
\begin{align}
    \forall v \in \cs: & & \inp{v}{\mu_\tau} &= \inp{v(t_0, \cdot)}{\mu_0} + \inp{\Lie v}{\mu}. \label{eq:martingale_occ_v}
\end{align}
Relation \eqref{eq:martingale_occ_v} is called Liouville equation for \acp{ODE} and is referred to as Dynkin's equation in the context of \acp{SDE}. We will also equivalently express \eqref{eq:martingale_occ_v} \rw{with an implicit $(\forall v \in \mathcal{C})$ quantification} in a shorthand form  by 
\begin{align}
    \mu_\tau = \delta_{t_0} \otimes \mu_0 + \Lie^\dagger \mu. \label{eq:martingale_occ}
\end{align}
\rw{Tuple of measures $(\mu_0, \mu_\tau, \mu)$ satisfying \eqref{eq:martingale_occ} are \textbf{relaxed occupation measures}. Not every relaxed occupation measure is a true occupation measure, Theorem 3.1 of \cite{cho2002linear} outlines conditions for which the relaxed occupation measure is supported on the graphs of stochastic trajectories.}








%% file: sections/acronym.tex
\begin{acronym}[WSOS]
\acro{BSA}{Basic Semialgebraic}



\acro{LP}{Linear Program}
\acroindefinite{LP}{an}{a}




\acro{ODE}{Ordinary Differential Equation}
\acroindefinite{ODE}{an}{a}

\acro{PSD}{Positive Semidefinite}



\acro{SDE}{Stochastic Differential Equation}
\acroindefinite{SDE}{an}{a}

\acro{SDP}{Semidefinite Program}
\acroindefinite{SDP}{an}{a}

\acro{SOS}{Sum of Squares}
\acroindefinite{SOS}{an}{a}

\acro{WSOS}{Weighted Sum of Squares}

\end{acronym}

%% file: sections/methods.tex
\section{Unsafety Problem and Linear Programs}
\label{sec:unsafe_lp}

This section will present \acp{LP} to compute worst-case probabilities of \ra{unsafety} of \ra{$\Lie$-generated} stochastic processes.
\vspace{-0.25cm}
\subsection{Problem Statement}

\rw{The focus of this paper is to perform risk analysis} of stochastic processes by computing the probability that trajectories will reach an unsafe set $X_u$. Trajectories evolve over a maximal time horizon of $T$ in a state space $X \subseteq \R^n$, beginning in an initial set $X_0 \subseteq X$ and possibly entering the unsafe set $X_u \subseteq X$. We consider stochastic processes described by a generator $\Lie$ (e.g., \iac{SDE}), for which $\rw{\mu_t}$ is the time-dependent state probability distribution of the process starting from $X_0$. The \ra{hitting} time from $X$ and $X_u$ respectively are $\tau_X = \inf\{ t : [x \in \partial X \mid \rw{\mu_t}]\}$ and $\tau_u = \inf\{ t : [x \in  X_u \mid \rw{\mu_t}]\}$, together forming the stopping time $\tau_m = \min(\rw{\tau_u}, \tau_X, \rw{T})$.
The worst-case probability of unsafety for $\Lie$-trajectories starting at a point $x_0 \in X$ within a time horizon of $[t_0, T]$ is:
\begin{subequations}
\label{eq:unsafe_prob_singlept}
\begin{align}
    P^*(t_0, x_0) = &  \prb_{\rw{\mu}_{\rw{T}}}[x \in X_u] \label{eq:unsafe_prob_objective} \\
        \ra{\text{s.t.} \qquad } & x(0) = x_0 \\
     & x(t) \ \text{follows} \  \Lie \quad  \forall t \in [t_0, \tau_m]. 
\label{eq:unsafe_prob_sde}
\end{align}
\end{subequations}
\rw{Trajectories following the stopping time $\ra{\tau_m}$ will halt upon reaching the unsafe set $X_u$, the boundary of the set $\partial X$, or when $T$ time units are elapsed (for whichever event comes first).}
The stopping time $\ra{\tau_m}$ in \eqref{eq:unsafe_prob_sde} also ensures that trajectories $x(t)$ remain in $X$ for all relevant $t$.
\rw{
The worst-case probability of unsafety for $\Lie$-trajectories over an initial set $X_0 \subseteq X$ is
\begin{subequations}
\label{eq:unsafe_prob_single}
\begin{align}
    P^*(t_0, X_0) = &\sup_{x_0 \in X_0} P^*(t_0, x_0), \label{eq:unsafe_prob_objective_set}\\
     \intertext{and the averaged worst-case probability of unsafety when starting at a probability distribution of initial conditions $\mu_0 \in \Mp{X_0}$ is}
    P^*(t_0, \mu_0) = &\textstyle \int_{X_0} P^*(t_0, x_0) d \mu_0(x_0). \label{eq:unsafe_prob_objective_avg}
\end{align}
\end{subequations}
For a single initial condition $x_0 \in X$, the single and averaged probabilities of unsafety $P(t_0, x_0)$ and $P(t_0, \delta_{x_0})$ are equal}
The initial set of $X_0$ is safe if $P^*(t_0, X_0)=0$ and is unsafe if $P^*(t_0, X_0)=1$. \rw{Similarly, the distribution of initial conditions $\mu_0$ is safe if $P^*(t_0, \mu_0) = 0$ and is unsafe if $P^*(t_0, \mu_0) = 1$.} Any other value $P^*(t_0, X_0), \rw{P^*(t_0, \mu_0)} \in (0, 1)$ returns a maximal probability of unsafety of the \rw{stochastic process}. \rw{The goals of this paper are to form convergent convex optimization problem whose objectives are equal to \eqref{eq:unsafe_prob_objective_set} and \eqref{eq:unsafe_prob_objective_avg}, and to find a continuous risk contour $v \in \mathcal{C}$ that will converge from above in an $L_1$ sense to the optimal map $P^*(t, x_0)$.  Level sets of any such risk function $v$ can therefore be used to upper-bound the risk of stochastic execution for $\Lie$ when starting at any $(t, x_0) \in [t_0, T]\times X$. }



\subsection{Assumptions}

We will posit the following assumptions throughout this paper \rw{in order to obtain nonconservative conditions for unsafety analysis via convex programming}:
\begin{enumerate}
    \item[A1] The state-sets $X_0, X_u, X$ are all compact.
    \item[A2] The time interval $[t_0, T]$ is compact.
    \item[A3] Trajectories stop upon the first exit from $X$ ($\tau_X \wedge T$).
    \item[A4] The test function set $\cs = \textrm{dom}(\Lie)$ satisfies $\cs \subseteq C([t_0, T] \times X)$
    with $1 \in \cs$ and $\Lie 1 = 0$.
    \item[A5]$\cs$ separates points and is multiplicatively closed.
    \item[A6] There exists a countable set $\{v_k\} \in \cs$ such that $\forall v \in \cs:$ $(v, \Lie v)$ is contained in the bounded pointwise closure of the linear span of $\{(v_k, \Lie v_k)\}$. 
\end{enumerate}

\rw{Assumptions A1-A2 enforce compactness. The stopping Assumption A3 ensures that trajectories never exit the space $X$ and then return to it later.}
Assumptions A4-A6 are the same as in Condition 1 of \cite{cho2002linear}. \rw{Assumptions A4-A6 are very general, as they hold for the generators of Markov, \ac{SDE}, and L\'{e}vy processes.}


\subsection{Measure Program}

We will pose a convex but infinite-dimensional \ac{LP} in order to upper-bound program \eqref{eq:unsafe_prob_single}. Such an \ac{LP} will be posed in terms of \rw{the following variables}:
\rw{
\begin{subequations}
\label{eq:meas_vars}
    \begin{align}
        \mu_0 &\in \Mp{X_0} & & \text{Initial} \\
        \mu &\in \Mp{[t_0, T] \times X} & & \text{Occupation} \\
        \mu_c &\in \Mp{[t_0, T] \times X} & & \text{Complement} \\
        \mu_u &\in \Mp{[t_0, T] \times X_u} & & \text{Unsafe}
    \end{align}
\end{subequations}
The sum $\mu_u + \mu_c$ will act as a terminal measure with free terminal time. These measure variables will be designed such that $(\mu_0, \mu_u + \mu_c, \mu)$ is a relaxed occupation measure to the stochastic process. The initial distribution $\mu_0$ is an optimization variable when solving the $X_0$ problem in \eqref{eq:unsafe_prob_objective_set}, whereas it is fixed to a given initial distribution when solving the averaged unsafety problem in \eqref{eq:unsafe_prob_objective_avg}. 
}

\rw{The core contribution of  our paper is the following statement of no-relaxation-gap:}
\begin{thm} \label{thm:no_relaxation}
    The following \ac{LP} will upper-bound the value at \eqref{eq:unsafe_prob_objective_set} as $p^*(t_0, X_0) \geq P^*(t_0, X_0)$ under assumption A3.
    \begin{subequations}
\label{eq:unsafe_meas}
\begin{align}
    p^*({t_0}, X_0) = &\sup \quad \inp{1}{\mu_{\rw{u}}}  \label{eq:unsafe_meas_obj}\\
    \ra{\text{s.t.} \qquad }& \mu_{\rw{u}} + \mu_c =  \delta_{t_0} \otimes \mu_0 + \Lie^\dagger \mu \label{eq:unsafe_meas_liou} \\
    & \inp{1}{\mu_0} = 1 \label{eq:unsafe_meas_prob}\\
    & \mu_0 \in \Mp{X_0} \\ 
    & \mu, \ \mu_c \in \Mp{[t_0, T] \times X} \label{eq:unsafe_meas_support_pre}\\ 
    & \mu_{\rw{u}} \in \Mp{[t_0, T] \times X_u}.    \label{eq:unsafe_meas_support}
\end{align}
    Under Assumptions A1-A6, the objectives of \eqref{eq:unsafe_meas} and \eqref{eq:unsafe_prob_sde} are equal $(p^*({t_0}, X_0) = P^*({t_0}, X_0))$.
\end{subequations}
\end{thm}


\begin{proof}
See Appendix \ref{app:no_relaxation_gap}.
\end{proof}

    

\begin{cor}
\rw{
    In the case where $\mu_0$ is prespecified, the following \ac{LP} with objective $p^*(t_0, \mu_0)$ will upper-bound $P^*(t_0, \mu_0)$ from \eqref{eq:unsafe_prob_objective_avg} under A3, and will have no relaxation gap with $p^*(t_0, \mu_0) = P^*(t_0, \mu_0)$ under A1-A6:
        \begin{subequations}
\label{eq:unsafe_meas_given}
\begin{align}
    p^*({t_0}, \mu_0) = &\sup \quad \inp{1}{\mu_{\rw{u}}}  \label{eq:unsafe_meas_given_obj}\\
   \ra{\text{s.t.} \qquad } & \mu_{\rw{u}} + \mu_c =  \delta_{t_0} \otimes \mu_0 + \Lie^\dagger \mu \label{eq:unsafe_meas_given_liou} \\
    & \mu, \ \mu_c \in \Mp{[t_0, T] \times X} \label{eq:unsafe_meas_given_support_pre}\\ 
    & \mu_{\rw{u}} \in \Mp{[t_0, T] \times X_u}.    \label{eq:unsafe_meas_given_support}
\end{align}
\end{subequations}
}
\end{cor}
\vspace{-1cm}
\begin{proof}
    \rw{The upper-bound and no-relaxation-gap \ra{results follow  from retracing} the steps \ra{used} in Appendix \ref{app:no_relaxation_gap} to prove Theorem \ref{thm:no_relaxation}. For the upper-bound, let $(\mu[x_0], \mu_c[x_0], \mu_u[x_0])$ denote measures obtained by the construction procedure from Appendix \ref{app:upper_bound_construction} given an initial point $x_0$ and a stopping time $t^*$. Feasible measures  $(\mu, \mu_c, \mu_u)$ for\eqref{eq:unsafe_meas_given_liou}-\eqref{eq:unsafe_meas_given_liou} can therefore be gathered by averaging $(\mu[x_0], \mu_c[x_0], \mu_u[x_0])$ over all $x_0 \in \mu_0$, in which the cost is preserved. For the no-relaxation gap, the tuple $(\mu_0, \mu_{\rw{u}} + \mu_c, \mu)$ remains supported on the graph of stochastic trajectories, and thus the proof steps of Theorem \ref{thm:no_relaxation} can be used without modification.}
    
\end{proof}


\subsection{Function Program}

\rw{This subsection will define \acp{LP} in a continuous function variable $v \in \mathcal{C}$ that are dual to \eqref{eq:unsafe_meas} and \eqref{eq:unsafe_meas_given}. We first present these statements of duality. We then will analyze properties of this solution $v$ (noting that it will upper-bound $P^*(t_0, x_0)$ in a pointwise sense) and outline convergence properties. This solution $v$ will also be called a `risk contour' due to this pointwise satisfaction. These risk contours will be members of a restricted class of functions:
\begin{align}
    \mathcal{R}=&  \begin{Bmatrix}
         & \Lie v(t,x) \leq 0 & \forall (t, x) \in [t_0, T] \times X \\ v \in \mathcal{C}\mid  & v(t, x) \geq 0 &  \forall (t, x) \in [t_0, T] \times X \\  & v(t, x) \geq 1 & \forall (t, x) \in [t_0, T] \times X_u
    \end{Bmatrix}. \label{eq:up_prob}
\end{align}
}
\rw{The set $\mathcal{R}$ depends on the sets $\mathcal{C}, [t_0, T], X$, and $X_u$.}

\begin{thm}
\label{thm:unsafe_cont}
\Iac{LP} in terms of an auxiliary function $v$ that forms a weak dual ($p^*({t_0}, X_0) \geq d^*({t_0}, X_0)$) to \eqref{eq:unsafe_meas} is
\begin{subequations}
\label{eq:unsafe_cont}
\begin{align}
    d^*({t_0}, X_0)    =& \inf_{\gamma \in \R, \ v \in \rw{\mathcal{R}}} \quad \gamma \label{eq:unsafe_cont_obj} \\
   \ra{\text{s.t.} \qquad } & \gamma \geq v(t_0, x) & & \forall x \in X_0. \label{eq:unsafe_cont_init}
\end{align}
\end{subequations}
Strong duality ($p^*(t_0, X_0)=d^*(t_0, X_0)$) will hold under Assumptions A1-A4.
\end{thm}
\begin{proof}
    See Appendix \ref{app:duality}.
\end{proof}

\begin{rem}
    Constraints \eqref{eq:unsafe_cont_init}-\eqref{eq:unsafe_cont_init} have the same form as a stochastic (time-dependent) barrier function at a fixed probability level $\gamma$  (from \cite{prajna2004stochastic}). Our work involves application of $\gamma$ as an optimization variable, as well as a proof of no relaxation gap in Theorem \ref{thm:no_relaxation} under A1-A6.
\end{rem}

\begin{cor}
The following \ac{LP} is weakly dual to \eqref{eq:unsafe_meas_given}, and possesses strong duality under A1-A4:
    \begin{subequations}
\label{eq:risk_cont}
\begin{align}
    d^*(t_0, \mu_0)    =&\inf_{v \in \rw{\mathcal{R}}} \textstyle\int_{X_0} v(t_0, x) d(x_0) \! \! \! \! \! \! \! \! \! \! \! \! \! \! \! \! \! \! \! \! \! \label{eq:risk_cont_obj} \\
  \ra{\text{s.t.} \qquad }  & \Lie v(t,x) \leq 0 & & \forall (t, x) \in [t_0, T] \times X. \label{eq:risk_cont_lie}
\end{align}
\end{subequations}
\end{cor}
\begin{proof}
    This strong duality holds by slight modification of the method used in Theorem \ref{thm:unsafe_cont} and Appendix \ref{app:duality}.
\end{proof}

\rw{We now outline properties of the risk contour $v$ functions found by solving problems \eqref{eq:unsafe_cont} and \eqref{eq:risk_cont}. The first property of $v$ to analyze is its boundedness and pointwise-upper-bounding of $P^*$}.



\begin{thm}
\label{thm:sublevel_1}
    Under \rw{A1-A6, every $v \in \mathcal{R}$ satisfies}
   $v(t_0, x_0) \geq P^*(t_0, x_0)$ for every initial condition $x_0 \in X$.
\end{thm}
\begin{proof}
See Appendix \ref{app:superlevel_risk}.
\end{proof}

\begin{prop}
\rw{Both $d^*(t_0, X_0)$ and $d^*(t_0, \mu_0)$ are upper-bounded by 1.}
\label{prop:upper_bound}
\end{prop}
\begin{proof}
    The function $v(t, x) = 1$ is a member of $\mathcal{R}$. The point $(v = 1, \gamma = 1)$ is feasible for \eqref{eq:unsafe_meas} with value $\gamma = 1$, and $(v = 1)$ is feasible for \eqref{eq:unsafe_meas_given} with value $\inp{1}{\mu_0} = 1$. 
\end{proof}

    Level sets of any $v \in \mathcal{R}$ can therefore be used to upper-bound the risk of stochastic execution for $\Lie$ when starting at any $(t, x_0) \in [t_0, T]\times X$. \rw{Solutions $v$ to \eqref{eq:risk_cont} will be maximally close in an $L_1(\mu_0)$ sense to the true unsafe probability map based on the following theorem:}




\begin{thm}
\label{thm:converge_sequence_init}
    Let $\{v_k\}_{k \geq 1}$ be a sequence of solutions \rw{such that each $v_k \in \mathcal{R}$} with $\int_{X_0} v_k(t_0, x_0)d\mu_0(x_0) \rightarrow d^*(t_0, \mu_0)$ (under A1-A6). Then $v_k$ converges in $L_1(\mu_0)$ to $\rw{P}^*$ as in 
    \begin{align}
        \int_X \left(\vphantom{\sum} v_k(t_0, x_0) - \rw{P}^*(t_0, x_0)\right) d \mu_0(x_0) \underset{k\to\infty}{\longrightarrow} 0. \label{eq:l1_converge}
    \end{align}
\end{thm}
\begin{proof}
See Appendix \ref{app:risk_convergence}.

  


    


    
\end{proof}
\begin{rem}
\rw{Choosing the reference distribution $\mu_0$ as the uniform distribution over the compact $X$ gives a practitioner most possible information about the structure of $P^*$.}
\end{rem}

%% file: sections/unsafe_sos.tex
\section{Unsafety Semidefinite Programs}
\label{sec:unsafe_sos}


\rw{The infinite-dimensional linear programs in \eqref{eq:unsafe_cont} and \eqref{eq:risk_cont} must be truncated into finite-dimensional optimization problems in order to admit tractable numerical computation. \ra{One such method to perform this truncation is the moment-\ac{SOS} hierarchy.} This section will start by reviewing the moment-\ac{SOS} hierarchy of \acp{SDP} for truncations of infinite-dimensional \ac{LP} with polynomially structured data, \ra{and then will apply the hierarchy} towards truncation of \eqref{eq:unsafe_cont} and \eqref{eq:risk_cont}.}

\subsection{Sum-of-Squares Methods}
\label{sec:sos}

\rw{Let} $c \in \R[x]_{\leq 2k}$ be a polynomial. The polynomial $c$ is nonnegative if $\forall x \in \R^n: c(x) \geq 0$. The polynomial $c$ is \ac{SOS} if there exists $N$ polynomials $q_j \in \R[x]$ such that $c(x) = \sum_{j=1}^N q_j^2(x)$. The cone of \ac{SOS} polynomials is $\Sigma[x]$, and its subset of \ac{SOS} polynomials with degree $\leq 2k$ is $\Sigma[x]_{\leq 2k}.$ The set of \ac{SOS} polynomials equals the cone of nonnegative polynomials only in the \rw{cases of univariate polynomials $(n=1)$, quadratics ($2k=2$), or bivariate quartics $(n=2, 2k=4)$) }\cite{hilbert1888darstellung}. 
Verification of polynomial nonnegativity is generically an NP-hard task, and \ac{SOS} polynomials \rw{vanish} in the set of nonnegative polynomials as $n$ and $d$ rise \cite{blekherman2006there}.

\ac{SOS} representations up to fixed degree can be conducted by formulating \acp{SDP}. To each  polynomial $c \in \R[x]_{\leq 2k}$, there exists a (typically nonunique) associated polynomial vector $v_k(x) \in (\R[x]_{\leq d})^s$ and a symmetric \textit{Gram} matrix $Q \in \psd^s$ such that $c(x) = v_k(x)^T Q v_k(x)$. One such choice of vectors $v_k(x)$ is the set of monomials in $x$ between degrees $0..k$, in which case $s = \binom{n+k}{k}$. A polynomial $c$ is \ac{SOS} if its Gram matrix $Q$ is also \ac{PSD} \cite{choi1995sums}.

\Iac{BSA} set $\K$ is a set described by a finite number of bounded-degree polynomial inequality $(N_g)$ and equality constraints $(N_h)$:
\begin{align}
    \K = \{x \in \R^n \mid  g_i(x) \geq 0, \ h_j(x) = 0, \forall i\in1..N_g, j\in1..N_h\}. \label{eq:bsa}
\end{align}
The class of \ac{WSOS} polynomials over $\K$ in \eqref{eq:bsa} is the set $\Sigma[\K]$ of polynomials $c$ that admit a description of
\begin{subequations}
\label{eq:psatz_noeps}
\begin{align}
c(x) = &\sigma_0(x) + \textstyle \sum_i {\sigma_i(x)g_i(x)} + \textstyle \sum_j {\phi_j(x) h_j(x)} \\
& \sigma_0, \sigma_i \in \Sigma[x], \quad \phi_j \in \R[x].
\end{align}
\end{subequations}
The set $\K$ is compact if there exists a finite $R>0$ such that $\{x \mid \norm{x}_2^2 \leq R\} \supseteq \K$. The set is additionally Archimedean if $R - \norm{x}_2^2 \in \Sigma[\K].$ Not every compact set is Archimedean \cite{cimpric2011closures}, but if an $R$ verifying compactness of $\K$ is known, then the constraint $R - \norm{x}_2^2 \geq 0$ can be added to the description of $\K$ to render $\K$ Archimedian. Every compact polytope and ellipsoid  is also Archimedean \cite{lasserre2009moments}.

The Putinar Positivestellensatz states that every positive polynomial over $\K$ is also a member of $\Sigma[\K]$ when $\K$ is Archimedean \cite{putinar1993compact}. However, the polynomial degrees of the multipliers $(\sigma, \phi)$ needed to certify nonnegativity of such a positive $c$ could be exponential in $n$ and $k$ \cite{nie2007complexity}.



\subsection{Sum-of-Squares Assumptions }
In order to apply \ac{SOS} methods towards convergent approximation of problems \eqref{eq:unsafe_cont} and \eqref{eq:risk_cont}, we require the additional assumptions of polynomial structure:
\begin{enumerate}
    \item[A7] The sets $X_0, \ X_u, \ X$ are all Archimedean \ac{BSA} sets.
    \item[A8] The generator $\Lie$ is closed under polynomials ($v(t, x) \in \R[t, x] \implies \Lie v(t, x) \in \R[t, x]$).
\end{enumerate}

Given a degree $k$ and a generator $\Lie$ obeying Assumption A8, we can define the dynamics degree $\tilde{k}$ as $\tilde{k} = \left \lceil\max_{v \in \R_{\leq 2k}[t, x]} \deg(\Lie v(t, x))/2 \right \rceil.$
    


\subsection{Unsafe-Probability SDP}

The degree-$k$ \ac{SOS} tightening of program \eqref{eq:unsafe_cont} is
\begin{subequations}
\label{eq:unsafe_cont_sos}
\begin{align}
    d^*_k(t_0, X_0)    =& \inf_{\gamma \in \R} \quad \gamma \label{eq:unsafe_cont_sos_obj} \\
   \ra{\text{s.t.} \qquad } & \gamma -  v(t_0, x) \in \Sigma[X_0]_{\leq 2k} \label{eq:unsafe_cont_sos_v_init}\\
    & -\Lie v(t,x) \in \Sigma[[0, T] \times X]_{\leq 2\tilde{k}} \label{eq:unsafe_cont_sos_v_lie}\\
    & v(t, x) \in \Sigma[[0, T] \times X]_{\leq 2k} \label{eq:unsafe_cont_sos_v_comp}\\
    & v(t, x) - 1\in \Sigma[[0, T] \times X_u]_{\leq 2k},\label{eq:unsafe_cont_sos_v_unsafe}
\end{align}
\end{subequations}
\rw{and the degree-$k$ \ac{SOS} tightening of program \eqref{eq:risk_cont} is}
\begin{subequations}
\label{eq:risk_sos}
\begin{align}
    d^*_k(t_0, \mu_0)    =& \inf \int_X v(t_0, x) d\mu_0(x) \\        
   \ra{\text{s.t.} \qquad } & -\Lie v(t,x) \in \Sigma[[0, T] \times X]_{\leq 2\tilde{k}} \label{eq:risk_sos_v_lie}\\
    & v(t, x) \in \Sigma[[0, T] \times X]_{\leq 2k} \label{eq:risk_sos_v_comp}\\
    & v(t, x) - 1\in \Sigma[[0, T] \times X_u]_{\leq 2k}. \label{eq:risk_sos_v_unsafe}
\end{align}
\end{subequations}


\begin{rem}
\rw{The objectives are always upper-bounded by \rw{$d_k^*(t_0, X_0) \leq 1$ and }$ d_k^*(t_0, \mu_0) \leq 1$ because $v(t, x) = 1$ is always feasible for their constraints (from Proposition \ref{prop:upper_bound}).}
\end{rem}
\begin{thm}
\label{thm:sos_converge}
       Under Assumptions A1-A8, program \eqref{eq:unsafe_cont_sos} will converge to \eqref{eq:unsafe_cont} with $\lim_{k\rightarrow \infty} d^*_k(t_0, X_0) = d^*(t_0, X_0)$, and \eqref{eq:risk_sos} will converge to \eqref{eq:risk_cont} as  $\lim_{k\rightarrow \infty} d^*_k(t_0, \mu_0) = d^*(t_0, \mu_0)$.
\end{thm}
\begin{proof}
    See Appendix \ref{app:moment_sos_converge}.
\end{proof}

Let $v_k(t, x) \in \Sigma[[0, T] \times X]_{\leq 2k}$ be the  solution to \eqref{eq:risk_sos} at degree-$k$, and let $I_{u}$ be the $0/1$ indicator function on the unsafe set $X_u$. Then the probability of unsafety when starting at a point $x_0 \in X$ is upper-bounded by
\begin{align}
\label{eq:parametric_subvalue}
    q_{1:k}(x) = \min(1, \min_{k' \in 1..k} v_{k'}(t_0, x)).
\end{align}

\begin{cor}
    The sequence of functions $q_{1:k}(x)$ in increasing $k$ will converge in measure $\mu_0$ to $v(t_0, x)$.
\end{cor}
\begin{proof}
    This corollary follows from Theorem \ref{thm:converge_sequence_init}, in which the sequence $\{v_{k'}\}_{\rw{k' \in 1\ldots k}}$ is used to approximate $v^*$. 
\end{proof}
\begin{rem}
    \rw{Let $v_1, v_2$ be polynomials found by solving \eqref{eq:risk_sos} at degrees  $k_1 \in k_2 \in 1..k$ with $k_1 \leq k_2$ with objectives $d_{k_1}(t_0, \mu_0) \geq d_{k_2}(t_0, \mu_0)$. It may not be true that $v_1 \geq v_2$ pointwise. Taking a minimum along all degrees in \eqref{eq:parametric_subvalue} when finding $q_{1:k}$ allows for a sharper estimate of $P^*$ from above than simply taking $v_k$ as the risk contour.  }
\end{rem}



\subsection{Computational Complexity}

We will quantify computational complexity of the degree-$k$ tightenings of \eqref{eq:unsafe_cont_sos} and \eqref{eq:risk_sos} by the size of the maximal Gram matrices involved in their \ac{SOS} program. In the typical case where $\tilde{k} > k$ (only violated under A1-A8 when $\Lie$ maps every polynomial to a constant), the largest Gram matrix will occur in the Lie constraints \eqref{eq:unsafe_cont_sos_v_lie} and \eqref{eq:risk_sos_v_lie}. The Lie constraints each have $n+1$ variables $(t, x)$, so the Gram matrix size when using the monomial basis is $\binom{n+1+\tilde{k}}{\tilde{k}}.$ All other constraints have a lower degree ($k$ rather than $\tilde{k}$), or are posed only over the $n$ variables $x$.
The complexity of using an interior-point method to solve the \ac{SOS} programs will therefore scale based on $O((n+1)^{6\rw{k}})$ for fixed $\rw{k}$ or $O(\rw{k}^{(n+1)})$ for fixed $n$ \cite{lasserre2009moments}.


%% file: sections/examples.tex
\section{Numerical Examples}

\label{sec:examples}
MATLAB (2021a) code to reproduce all examples is available at \url{https://github.com/jarmill/prob_unsafe}. All programs are modeled using YALMIP \cite{lofberg2004yalmip} and solved using Mosek 10 \cite{mosek92}. All examples will involve an initial time of $t_0 = 0$. \rw{Experiments were run on an Intel Core i9 2.30GHz with 16.0 GB of RAM allocated to MATLAB.}
\subsection{Two-Dimensional Cubic SDE}
\label{sec:example_cubic}
Our first demonstration analyzes safety of a cubic polynomial \ac{SDE} from Example 1 of \cite{prajna2004stochastic}:
\begin{equation}
\label{eq:flow_sde}
    dx = \begin{bmatrix}x_2 \\ -x_1 -x_2 - \frac{1}{2}x^3_1\end{bmatrix}dt + \begin{bmatrix} 0 \\ 0.1 \end{bmatrix}dW.
\end{equation}
Safety of \eqref{eq:flow_sde} is evaluated within the state space of $X = [-2, 2]^2$ until a time horizon of $T=5$. The unsafe set is a moon-shaped region $X_u = \{x \in \R^2 \mid 0.9083^2 \leq (x_1 + 0.5006)^2 + (x_2 + 0.2902)^2, 0.5^2 \geq (x_1 -0.2)^2 + x_2^2 \}.$ The initial set $X_0$  is a circle of radius $R_0$ and center $[0.85; -0.75]$. Figure \rw{\ref{fig:flow_traj}} plots trajectories of \eqref{eq:flow_sde} starting from $X_0$ (magenta region with $R_0=0.2$). Note how some sampled trajectories touch the right corner of the red half-circle $X_u$ (and are therefore unsafe). Table \ref{tab:unsafe_flow} reports the probability of unsafety computed by program \eqref{eq:unsafe_cont_sos} for a circular initial set of radius $R_0$.
\begin{figure}[!h]
    \centering    
    \includegraphics[width=0.5\linewidth]{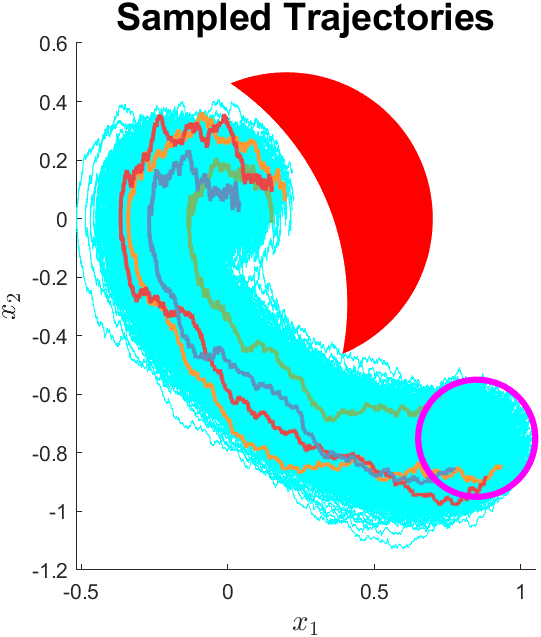}
    \caption{
    Trajectories of \ac{SDE} \eqref{eq:flow_sde} (cyan) initialized in a disk $X_0$ (purple boundary), hitting an unsafe region $X_u$ (red moon).
    }
    \label{fig:flow_traj}
\end{figure}



\begin{table}[!h]
\centering
   \caption{Unsafe probability upper-bounds for system \eqref{eq:flow_sde} \label{tab:unsafe_flow}}
\begin{tabular}{lllllll}
\multicolumn{1}{c}{order} & \multicolumn{1}{c}{1} & \multicolumn{1}{c}{2} & \multicolumn{1}{c}{3} & \multicolumn{1}{c}{4} & \multicolumn{1}{c}{5} & \multicolumn{1}{c}{6} \\ \hline
$R_0=0$    &     1& 0.9442& 0.6072& 0.4805& 0.4068& 0.3696           \\
$R_0=0.2$  &     1& 0.9736& 0.7943& 0.7091& 0.6390& 0.6132
\end{tabular}
\end{table}


Figure \ref{fig:risk_circ} plots the {unsafety} upper-bounding function $\min(1, v(0, x))$ {at $T=5$} acquired by solving \eqref{eq:unsafe_cont_sos} at order {$k=$}6 with $R_0=0.2$. The unsafe set $X_u$ is the red moon. The magenta circle is {the boundary of} the initial set $X_0$. Note how the probability estimate is sharper in the region surrounding $X_0$, as compared to the sea of $\prb=1$ saturation away from $X_0$.

In contrast, Figure \ref{fig:risk_X} solves the order-6 \ac{SOS} tightening of program \eqref{eq:risk_cont} to produce a risk map $v$ valid in $X = [-2, 2]^2$. This risk map results in $v(0, [0.85; -0.75]) = 0.4366,$ which is looser than the order-6 probability bound of $0.3696$ from Table \ref{tab:unsafe_flow}. However, Figure \ref{fig:risk_X} produces an interpretable visualization of risk across the entire set $X$.

\begin{figure}[!h]
     \centering
     \begin{subfigure}[b]{0.49\linewidth}
         \centering
         \includegraphics[width=\textwidth]{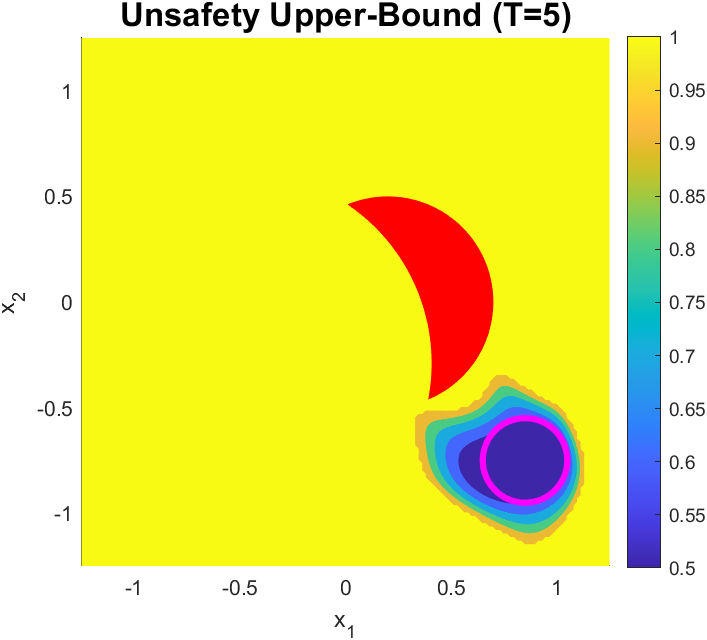}
         \caption{{solution of} \eqref{eq:unsafe_cont_sos}}
         \label{fig:risk_circ}
     \end{subfigure}
     \begin{subfigure}[b]{0.49\linewidth}
         \centering
         \includegraphics[width=\textwidth]{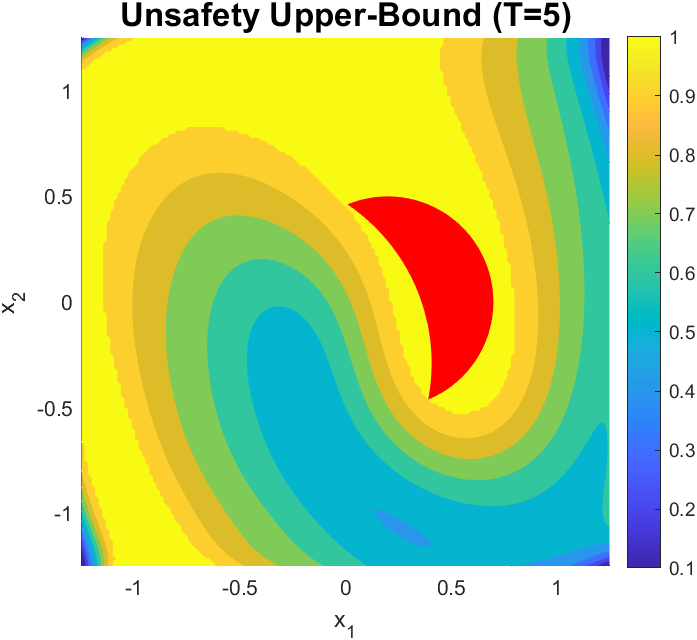}
         \caption{{solution of} \eqref{eq:risk_sos}}
         \label{fig:risk_X}
     \end{subfigure}
        \caption{Risk levels of {the unsafety upper-bound function} $v(0, x)$ at {$t=T=5$} and order $k=$6 for {SDE \eqref{eq:flow_sde}
      initialized in a disk $X_0$ (purple boundary) with unsafe region $X_u$ (red moon).}}
        \label{fig:risk_cubic}
\end{figure}



\subsection{Four-Dimensional SDE}

This example is inspired by dynamics of floating offshore platforms (based on the two-dimensional system of Example 5 of \cite{xue2023reach}). Two platforms are connected together through a restorative force (nonlinear spring), and the second platform is attracted to the point $x=0$ (by a buoy/anchor). The platforms experience a stochastic force that models the buffeting by waves.
The positions of the two platforms are $(x_1, x_2)$ and their velocities are $(\dot{x}_1, \dot{x}_2)$. 
\ra{The nonlinear Hooke force of the connecting spring is}
\begin{align}
    F_{12}(x) &= ((x_2 - x_1) - 0.5) - 0.05((x_2 - x_1) - 0.5)^3. \nonumber
\end{align}
The dynamics of this system are
\begin{align}
    d x_1 &= \dot{x}_1 dt \qquad \qquad \qquad 
    d x_2 = \dot{x}_2 dt  \label{eq:platform}\\
    d \dot{x}_1 &=  (F_{12}(x)- \dot{x}_1)dt +(-\dot{x}_1 + 1)dW_1 \nonumber\\    
    d \dot{x}_2 &= (-F_{12}(x) - \dot{x}_2 - x_2)dt +(-\dot{x}_1 + 1)dW_2.\nonumber
\end{align}
The state of the system is $x = (x_1, x_2, \dot{x}_1, \dot{x}_2)$\ra{.}
The system in \eqref{eq:platform} is executed in the space $X = \{x \in \R^4 \mid (x_1, x_2) \in [-2, 2]^2, \ (\dot{x}_1, \dot{x}_2) \in [-1, 1]^2\}$ for a time horizon of $T=8$. The platforms begin their execution at the point $x = (-0.25, 0.5, 0.2, 0.1)$. Figure \ref{fig:platform_state} plots 300 random realizations of the positions and velocities (in phase space) of the two platforms following dynamics \eqref{eq:platform} starting from initial condition (magenta dots). The left (blue) trajectories are the behavior of platform 1 $(x_1, \dot{x}_1)$, whereas the right (red) trajectories are the response of platform 2 $(x_2, \dot{x}_2)$.
The unsafe set is the region $X_u = \{x \in \R^4 \mid (x_1-x_2) ^2 \leq 0.15^2\}$, describing the potential for the platforms to crash into each other. Figure \ref{fig:platform_diff} plots the difference in platform positions $x_2(t) - x_1(t)$ as a function of time $t$ between $t \in [0, T].$ The red dotted line in Figure \ref{fig:platform_diff} denotes the beginning of unsafety with the level set $x_2(t) - x_1(t) = 0.15$. Some empirically sampled trajectories in Figure \ref{fig:platform_diff} cross this red line, signalling that the initial point $x_0$ has the potential to enter the unsafe region.

\begin{figure}[!h]
     \centering
     \begin{subfigure}[b]{0.49\linewidth}
         \centering
         \includegraphics[width=\linewidth]{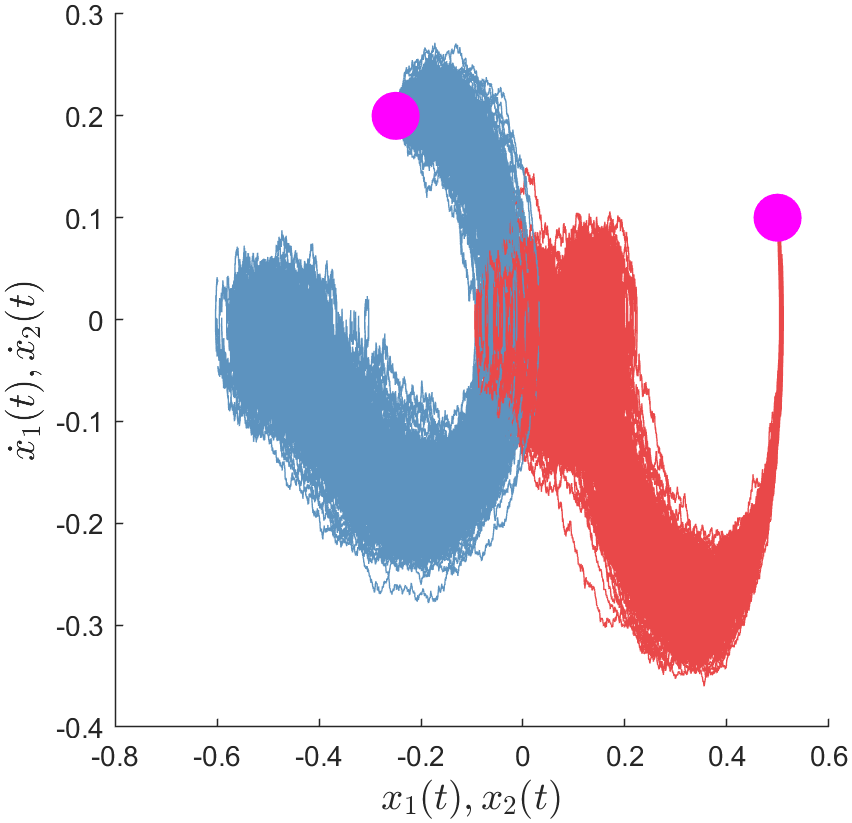}
         \caption{Velocity vs. position}
         \label{fig:platform_state}
     \end{subfigure}
     \hfill
     \begin{subfigure}[b]{0.49\linewidth}
         \centering
         \includegraphics[width=\linewidth]{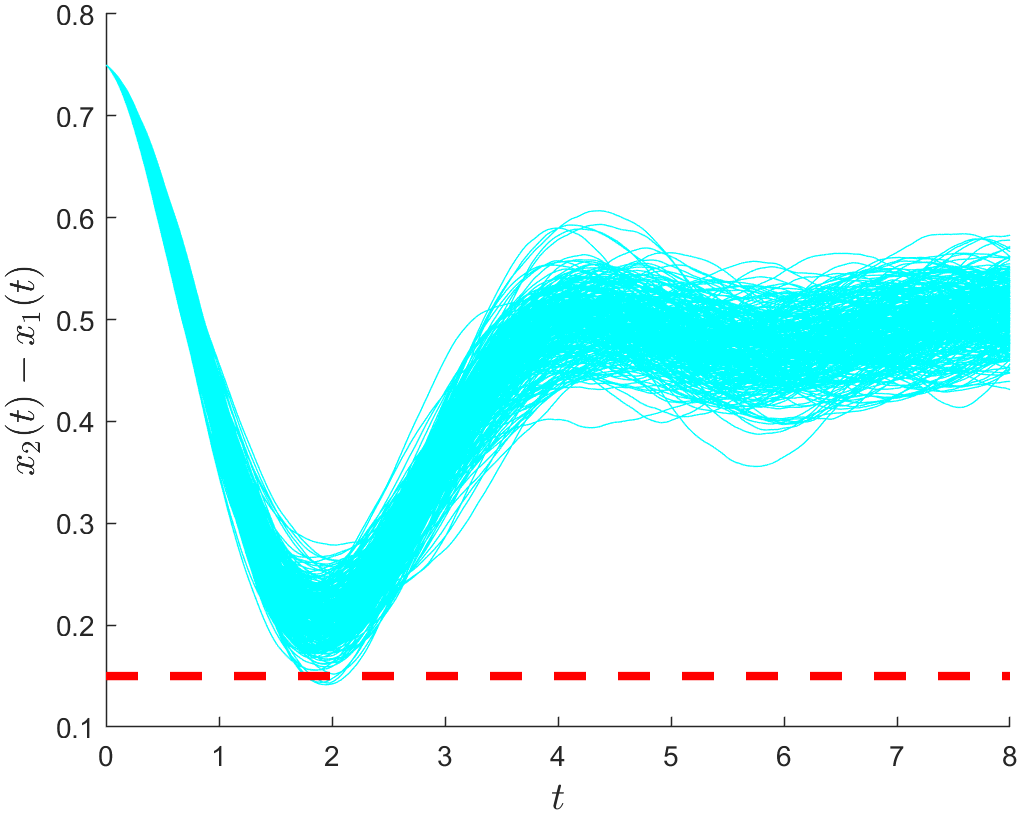}
         \caption{Difference in  positions}
         \label{fig:platform_diff}
     \end{subfigure}     
        \caption{\rw{Positions of the floating platforms from \eqref{eq:platform}}}
        \label{fig:platform}
\end{figure}

Table \ref{tab:unsafe_platform} lists upper bounds for the probability of entering $X_u$ starting from $x_0$ as computed through finite-degree truncations in \eqref{eq:unsafe_cont_sos}. 
\begin{table}[!h]
\centering
\rw{
   \caption{Unsafe probability upper-bounds for system \eqref{eq:platform} \label{tab:unsafe_platform}}
\begin{tabular}{lllllll}
\multicolumn{1}{c}{order} & \multicolumn{1}{c}{1} & \multicolumn{1}{c}{2} & \multicolumn{1}{c}{3} & \multicolumn{1}{c}{4} & \multicolumn{1}{c}{5} & \multicolumn{1}{c}{6} \\ \hline
$x_0$  &     1& 0.9236& 0.3068& 0.1344& 0.1162& 0.1121        \\
\end{tabular}
}
\end{table}

\subsection{Discrete-time example}

This subsection will focus on a discrete-time stochastic process involving a time-step of $\tau = 1$. The law of this stochastic process with parameter $\lambda \in \R$ is
\begin{align}
    x_+ = \begin{bmatrix}
        -0.3 x_1  + 0.8x_2 + x_1 x_2 \lambda/4\\ -0.9x_1 -0.1x_2 - 0.2 x_1^2 +  \lambda/40
    \end{bmatrix},\label{eq:scatter_discrete}
\end{align}
in which $\lambda$ is i.i.d. sampled  according to a unit-normal distribution at each time $(\lambda_t \sim \mathcal{N}(0, 1))$.

We evaluate the safety of \eqref{eq:scatter_discrete} with respect to the state set $X = [-1.5, 1.5]^2$, the time horizon of $T = 10$, and the half-circle unsafe set of $X_u = \{x \mid 0.4^2 \geq (x_1 - 0.8)^2 + (x_2 - 0.2)^2, \ x_1 + x_2 \geq 1\}$. Stochastic trajectories  of \eqref{eq:scatter_discrete} evolve starting at an circular initial set with radius $R_0$ and center $[-1; 0]$. 
Figure \ref{fig:discrete_sample} plots 5,000 sampled trajectories (blue dots) with respect to the magenta  circular initial set and the red half-circle unsafe set. Some of the sampled trajectory points fall inside the red half-circle, corresponding to unsafety when beginning in $X_0$.

\begin{figure}[h]
    \centering
    \includegraphics[width=0.8\linewidth]{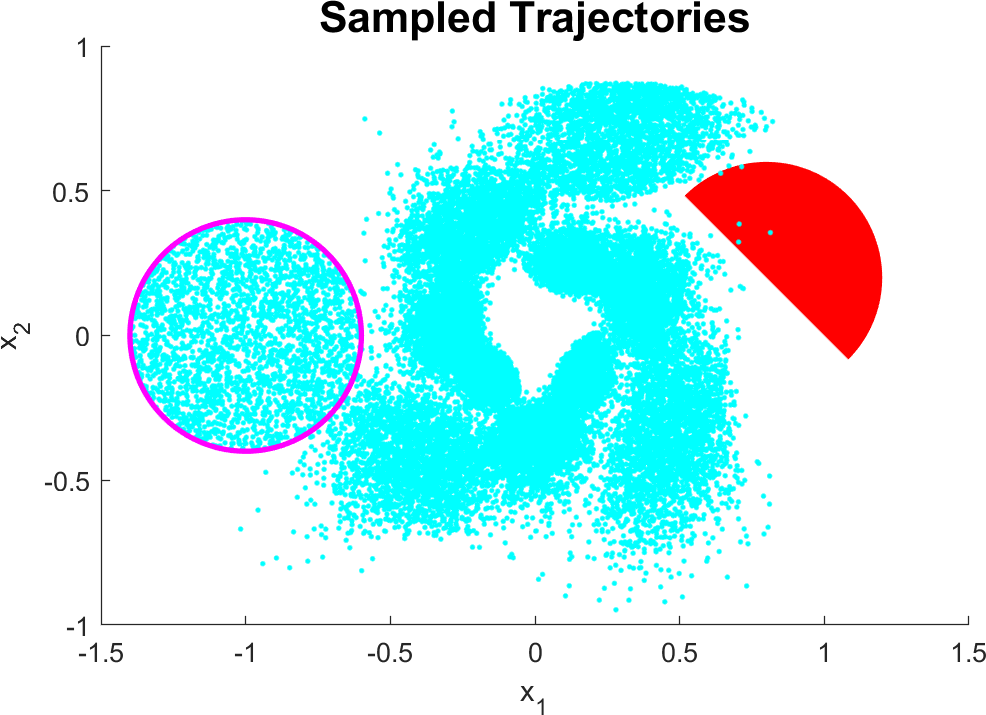}
    \caption{
Trajectories of discrete-time \eqref{eq:scatter_discrete} (cyan) initialized in a disk $X_0$ (purple boundary), hitting an unsafe region $X_u$ (red half disk).  
    }
    \label{fig:discrete_sample}
\end{figure}

Table \ref{tab:unsafe_discrete} reports probabilities of unsafety for \eqref{eq:scatter_discrete} found by solving \eqref{eq:unsafe_cont_sos} at initial radii $R_0 = 0$ and $R_0 = 0.4$. To improve numerical conditioning, we normalize the time-steps from $(\tau, T) = (1, 10)$ to $(\tau, T) = (0.1, 1)$ without affecting the autonomous dynamics in \eqref{eq:scatter_discrete}.

\begin{table}[!h]
\centering
   \caption{Unsafe probability upper-bounds for system \eqref{eq:scatter_discrete} \label{tab:unsafe_discrete}}
\begin{tabular}{lllllll}
\multicolumn{1}{c}{order} & \multicolumn{1}{c}{1} & \multicolumn{1}{c}{2} & \multicolumn{1}{c}{3} & \multicolumn{1}{c}{4} & \multicolumn{1}{c}{5} & \multicolumn{1}{c}{6} \\ \hline
$R_0=0$    &     1& 1& 0.1569& 0.0103& 1.871$\mathbf{e}$-3& 7.052$\mathbf{e}$-4          \\
$R_0=0.4$  &     1& 1& 0.9801& 0.7054& 0.5225& 0.4017
\end{tabular}
\end{table}

Figure \ref{fig:risk_circ_discrete} plots risk contours found by solving \eqref{eq:unsafe_cont_sos} at order 6 for $R_0 =0.4$.
 

Figure \ref{fig:risk_discrete} plots risk contours $v(0, x)$ acquired from \eqref{eq:risk_sos} at degree {$k=$}6  {and time $t=T=10$.} The returned risk map has an evaluation of $v(0, [-1; 0]) = 0.4915$, as compared to the \eqref{eq:unsafe_cont_sos} estimate of $7.052\times 10^{-4}$ at $R_0=0$.


\begin{figure}[!h]
     \centering
     \begin{subfigure}[b]{0.49\linewidth}
         \centering
         \includegraphics[width=\textwidth]{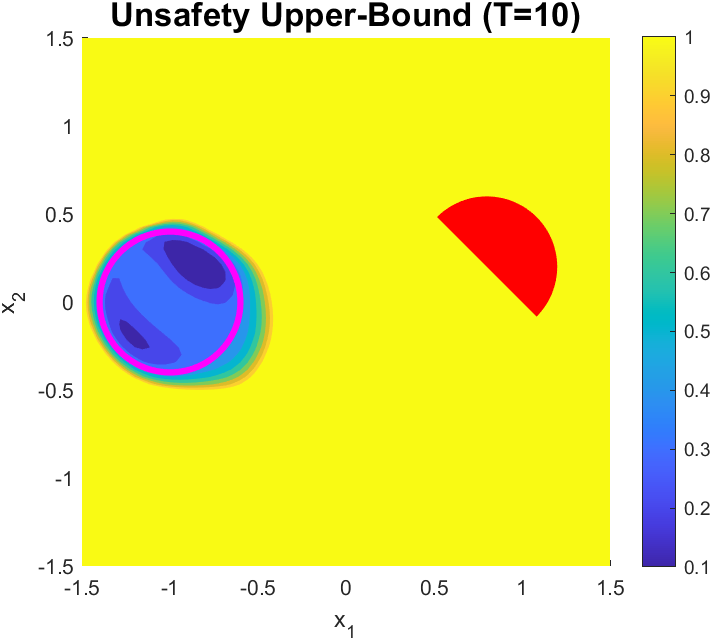}
         \caption{{solution of} \eqref{eq:unsafe_cont_sos}}
         \label{fig:risk_circ_discrete}
     \end{subfigure}
     \begin{subfigure}[b]{0.49\linewidth}
         \centering
         \includegraphics[width=\textwidth]{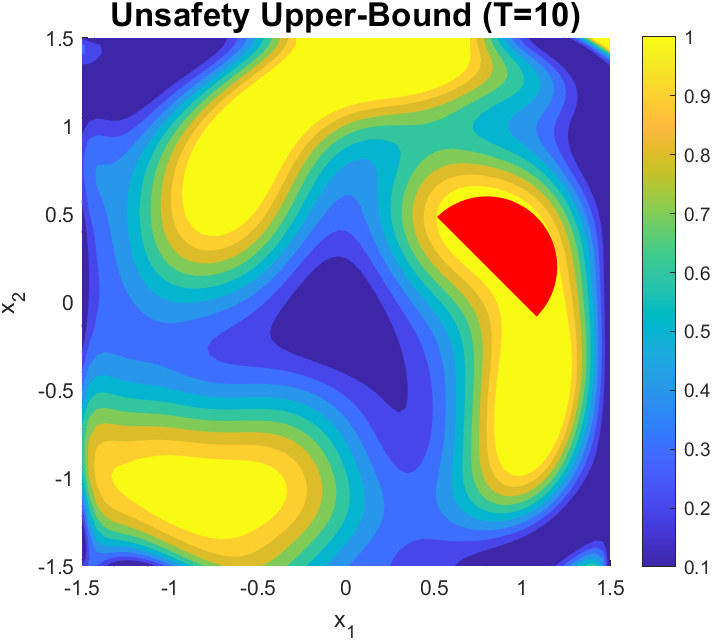}
         \caption{{solution of} \eqref{eq:risk_sos}}
         \label{fig:risk_discrete}
     \end{subfigure}
              \caption{Risk levels of {unsafety upper-bound function} $v(0, x)$ at {$\rw{T}=10$} and order {$k=$}6 for {SDE \eqref{eq:scatter_discrete}
      initialized in a disk $X_0$ (purple boundary) with unsafe region $X_u$ (red moon).}
        \label{fig:risk_discrete_all}}
\end{figure}


%% file: sections/conclusion.tex
\section{Conclusion}

\label{sec:conclusion}

This paper presents a method to analyze the probability of unsafety for stochastic processes by forming \acp{LP} in occupation measures. These \acp{LP} are upper-bounded by the moment-\ac{SOS} hierarchy of \acp{SDP}, yielding a convergent sequence of bounds to the true probability of unsafety. Modification of the objective leads to the development of visually interpretable risk contours for use in analysis and motion planning. The presented method can be used for any stochastic process satisfying Assumptions A1-A8.
Future work involves finding redundant constraints to refine and improve the unsafe-probability bounds. Other work includes  \rw{ developing convex reformulations of  risk-sensitive stochastic control methods}.

%% file: sections/acknowledgements.tex
\section*{Acknowledgements}

The authors would like to thank Roy S. Smith, Niklas Schmid, and the Automatic Control Laboratory at ETH Z\"{u}rich for their advice and support.


%% file: appendix/app_no_relaxation_measure.tex
\section{Proof of No-Relaxation-Gap}

\label{app:no_relaxation_gap}

\rw{This appendix will contain the proof of Theorem \ref{thm:no_relaxation}. It will start by proving that $p^*(t_0, x_0) \geq P^*(t_0, x_0)$ under Assumption A3 (Upper-Bound), and continue with a proof that $p^*(t_0, x_0) = P^*(t_0, x_0)$ under Assumptions A1-A6.}

\subsection{Upper-Bound}
\label{app:upper_bound_construction}
\rw{Let $x_0 \in X_0$ be an initial condition and $t^* \in [t_0, T]$ be a stopping time. A set of measures $(\mu_0, \mu, \mu_c, \mu_u)$ from \eqref{eq:meas_vars} may be constructed to fulfill constraints \eqref{eq:unsafe_meas_liou}-\eqref{eq:unsafe_meas_support} for each $x_0$. Because this construction is an injective map between $x_0$ and measures such that the cost is preserved, it would therefore hold that $p^*(t_0, X_0) \geq P^*(t_0, X_0)$.

To begin, the initial measure $\mu_0$ may be chosen as the Dirac delta $\delta_{x = x_0}$. The measure $\mu$ can be set to the occupation measure of the stochastic trajectory $x(t \mid x_0)$ between the initial time of ${t_0}$ and the stopping time of $\rw{\min(\tau_X, \tau_{u}, T)}$. 
The time-state distribution at time $T$ is a terminal measure $\mu_{T}$. The unsafe measure $\mu_u$ can be chosen as the restriction of the terminal measure to $[0, T] \times X_u$ (trajectories that stopped according to $\tau_u$), whereas $\mu_c$ is selected as a measure such that $\mu_{T} = \mu_u + \mu_c$. The mass $\inp{1}{\mu_u}$ in the objective \eqref{eq:unsafe_meas_obj} is then equal to $\text{Prob}_{\mu_{T}}[x \in X_u]$ for this trajectory.}
Because there exists a mapping from every  feasible point $x_0$ of \eqref{eq:unsafe_prob_single} to a set of  measures in \rw{\eqref{eq:unsafe_meas_liou}-\eqref{eq:unsafe_meas_support_pre} that preserves the cost (probability of unsafety)}, it holds that $\rw{p^*}({t_0}, X_0) \geq P^*({t_0}, X_0)$.    
\subsection{No-Relaxation-Gap}
\label{app:no_gap_main}
        By Theorem 3.1 of \cite{cho2002linear} (and under Assumptions A1-A6), every  tuple $(\mu_0, \mu, \mu_\tau)$ with  $\mu_0 \in \Mp{[t_0, T] \times X}$  (with $\inp{1}{\mu_0}=1$) and $\mu, \mu_\tau \in \Mp{[t_0, T] \times X}$ satisfying a Martingale relation \eqref{eq:martingale_occ} 
    is supported on the graph of a stochastic process \eqref{eq:unsafe_prob_sde}.

    Given feasible measures $(\mu_0, \mu, \mu_{\rw{u}}, \mu_c)$ satisfying \rw{\eqref{eq:unsafe_meas_liou}-\eqref{eq:unsafe_meas_support_pre}}, 
   it holds that the  tuple $(\mu_0, \mu, \mu_{\rw{u}} + \mu_c)$ 
    satisfies the Martingale relation \eqref{eq:martingale_occ}. The feasible measures are therefore supported on the graph \ra{of}    
    a stochastic process.
    \rw{Because every relaxed occupation measure $(\mu_0, \mu, \mu_{\rw{u}} + \mu_c)$ corresponds to a distribution of true occupation measures of the stochastic trajectory, there is therefore no gap between $P^*(t_0, x_0)$ (over stochastic trajectories) and $p^*(t_0, x_0)$ (over measures).} $\qed$

%% file: appendix/app_duality.tex
\section{Strong Duality}

\label{app:duality}

This appendix will prove strong duality between \eqref{eq:unsafe_meas} and \eqref{eq:unsafe_cont} by using arguments from Theorem 2.6 of  \cite{tacchi2021thesis}.
We collect the variables in \eqref{eq:unsafe_meas} and \eqref{eq:unsafe_cont} respectively into,
\begin{align}
    \bbmu &= (\mu_0, \mu, \mu_c, \mu_{\rw{u}}) &   \bell &= (\gamma, v). \label{eq:bbmu}
\end{align}
Variable spaces related to $\bbmu$ are,
\begin{align}
    \mathcal{X}' &= C(X_0) \times  C([0, T]\times X)^2 \times C([0, T] \times X_u) \label{eq:dual_spaces}\\
    \mathcal{X} &= \mathcal{M}(X_0) \times \mathcal{M}([0, T]\times X)^2 \times \mathcal{M}([0, T] \times X_u),\nonumber
\end{align}
with nonnegative subcones of
\begin{align}
    \mathcal{X}_+' &= C_+(X_0) \times C_+([0, T]\times X)^2 \times C_+([0, T] \times X_u) \nonumber \\
    \mathcal{X}_+ &= \Mp{X_0} \times \Mp{[0, T]\times X}^2 \times \Mp{[0, T] \times X_u}. \nonumber
\end{align}
Corresponding spaces to $\bell$ are 
\begin{align}
    \mathcal{Y}' &= {\cs \times \R} & 
    \mathcal{Y} &= {\cs' \times 0}.
\end{align}
Following the notation of \cite{tacchi2021thesis}, we write that $\mathcal{Y}_+ = \{0_\mathcal{Y}\}$ and $\mathcal{Y}_+' = \mathcal{Y}'$. We note the containments of  $\bbmu \in \mathcal{X}_+, \ \bell \in \mathcal{Y}', \ $ and also  that $(\mathcal{X}, \mathcal{X}')$ form a pair of topological dual spaces under Assumption A1. The topologies for the spaces $\mathcal{X}$ and $\mathcal{Y}'$ are the weak-* topology and sup-norm-bounded weak topology respectively.

An adjoint pair of affine maps $\A$ and $\A'$  may be defined as
\begin{align}
    \mathcal{A}(\bbmu) &= [\mu_{\rw{u}} + \mu_c - \Lie^\dagger \mu - \delta_{t_0} \otimes \mu_0, \ \inp{1}{\mu_0}] \\
    \mathcal{A}^*(\bell) &= [\gamma - v(t_0, \bullet), \ -\Lie_v, \ v, \ v],\nonumber
\end{align}
with vectors describing the cost and constraint terms as
\begin{subequations}
\label{eq:cost_constraint}
    \begin{align}
        \mathbf{b} &= [0, \ 1]  & 
        \mathbf{c} &= [0, \ 0, \ 0, \ 1].
    \end{align}
\end{subequations}
Pairings with the vectors in \eqref{eq:cost_constraint} satisfy\begin{subequations}
\begin{align}
\inp{\mathbf{c}}{\boldsymbol{\mu}} &= \inp{1}{\mu} & 
    \inp{\boldsymbol{\ell}}{\mathbf{b}} &= \gamma.
\end{align}
\end{subequations}
Problem \eqref{eq:unsafe_meas} may be expressed  as 
\begin{align}
    p^*(t_0, X_0) =& \sup_{\boldsymbol{\mu} \in \mathcal{X}_+} \inp{\mathbf{c}}{\boldsymbol{\mu}} & & \mathbf{b} - \A(\boldsymbol{\mu}) \in \mathcal{Y}_+. \label{eq:unsafe_meas_std}\\
\intertext{Similarly, the function \ac{LP} in  \eqref{eq:unsafe_cont} can be expressed as}
    d^*(t_0, X_0) = &\inf_{\boldsymbol{\ell} \in \mathcal{Y}'_+} \inp{\boldsymbol{\ell}}{\mathbf{b}}
    & &\A'(\boldsymbol{\ell}) - \mathbf{c} \in \mathcal{X}_+. \label{eq:unsafe_cont_std}
\end{align}
Sufficient conditions to prove strong duality between \eqref{eq:unsafe_meas_std} and \eqref{eq:unsafe_cont_std} (by Theorem 2.6 of \cite{tacchi2021thesis}) are:

\begin{enumerate}
    \item[\rw{R1}] All feasible measures $\bbmu \in \mathcal{X}_+$ with $b - \mathcal{A}(\bbmu) \in \mathcal{Y}_+$ are bounded, and there exists such a feasible $\bbmu$.
    \item[\rw{R2}] All functions used to define $\mathbf{c}, \  \mathbf{b}, \ \mathcal{A}$ are continuous.
\end{enumerate}
Boundedness of measures in R1 is proven in Lemma \ref{lem:bounded_meas}, and feasibility of a measure solution is demonstrated by the construction process used in Appendix \ref{app:upper_bound_construction}. For R2, we note that both $\mathbf{c}$ and $\mathbf{b}$ in \eqref{eq:cost_constraint} are constant and are therefore continuous. Additionally, the set $\cs$ is specifically chosen as the preimage of continuous functions under $\Lie$. R1 and R2 are both satisfied under  Assumptions A1-A6, thus proving strong duality.


%% file: appendix/app_risk_upper_bound.tex
\section{Superlevel Approximation of Risk Map}

\label{app:superlevel_risk}

This appendix provides the proof of Theorem \ref{thm:sublevel_1}.


\subsection{Supermartingale Property}

We first introduce the notion of supermartingales before treating the upper-bounding of $P^*$, and then note that any feasible \rw{$v \in \mathcal{R}$ from \eqref{eq:up_prob}} is a supermartingale.

\begin{defn}
    A process $\{\rw{\nu}_t\}$ is a  \textit{supermartingale} if $\E[\rw{\nu}_{t + \Delta t} \mid \{\rw{\nu}_t\}] \leq \rw{\nu}_t$ \cite{williams1991probability}.
\end{defn}

\begin{prop}
\label{prop:supermartingale}
    Any $v \rw{\in \mathcal{R}} $ is a \textit{supermartingale} for the process $\Lie$ \cite{rogers2000diffusions}.  
\end{prop}
\begin{proof}
    Let $\{\rw{\mu_t}\}$ be the state-dependent distribution for a trajectory of the stochastic process $\Lie$ in \eqref{eq:unsafe_prob_sde}, respecting exit time $\rw{\min (\tau_X, \tau_u, T)}$. This set of probability distributions satisfies the Martingale property of \eqref{eq:martingale}. Given that $\Lie v \leq 0$ from \eqref{eq:risk_cont}, it holds that $\E[\Lie_0 v(t, x) \mid \rw{\mu}_{s'}] \leq 0$ for every stopping time $s'$ adapted to $\rw{\min (\tau_X, \tau_u, T)}$. The supermartingale relation $ \E[v(t+s, x) \mid \rw{\mu}_{t + s} ] \leq v(t, x) $ therefore holds for $v$.
\end{proof}

\begin{lem}
\label{lem:prob_bound}
    Let $v(t, x) \in \cs$ be a nonnegative function over $[t_0, T] \times X$ and also form a supermartingale with respect to $\Lie$ and $\rw{\mu_t}$. For a value $\lambda \geq 0$, an initial point $x_0 \in X$, and a $\Lie$-trajectory $x(t)$ starting from $x_0$ at time $t_0$, the following inequality holds:
    \begin{align}
        \textrm{Prob}_{\{\rw{\mu_t}\}}\left[\sup_{t \in [t_0, T]} v(t, x(t)) \geq \lambda \right] \leq v(t_0, x_0)/\lambda.  \label{eq:doob}     
    \end{align}
\end{lem}
\begin{proof}
    \rw{This statement follows by using the steps of the proof of Lemma 6 of \cite{prajna2004stochastic}, with the modifications of} stopping at time $T$ and allowing for a time-dependent $v(t, x)$.
\end{proof}

\subsection{Superlevel Property}

Theorem \ref{thm:sublevel_1} can now be proven based on arguments from Theorem 7 of \cite{prajna2004stochastic}.
The unsafe set $X_u$ is inside the 1-superlevel set of $v(t, x)$ at every time $t \in [t_0, T]$ \rw{by definition in $\mathcal{R}$}. 
This superlevel relation implies that 
    \begin{align}
       P(t_0, x_0)= &\prb_{\{\rw{\mu}_{\rw{T}}\}}\left[ x(t^*) \in X_u \mid x_0 \right]  \\
        &  \leq \prb_{\{\rw{\mu_t}\}}\left[\sup_{t \in [t_0, T]} v(t, x(t)) \geq 1 \mid x_0 \right]. \nonumber
    \end{align} Through Lemma \ref{lem:prob_bound} with $\lambda=1$, it holds that \begin{equation}
        \prb_{\{\rw{\mu_t}\}}\left[\sup_{t \in [t_0, T]} v(t, x(t)) \geq 1 \mid x_0 \right] \leq v(t_0, x_0),
    \end{equation} thus proving the theorem.

%% file: appendix/app_risk_convergence.tex
\section{Risk Contour Convergence}
\label{app:risk_convergence}

This appendix will provide a proof of Theorem \ref{thm:converge_sequence_init}.

  We first notice that
    \begin{align*}
        0 \leq p^*(t_0,x_0) & = \sup \{\inp{1}{\mu_p} \; | \; \mu_p + \mu_c = \delta_{(t_0,x_0)} + \Lie^\dagger \mu\} \\ & = \inf \{v(t_0,x_0) \; | \; \Lie v \leq 0 \; ; \; v \geq I_{X_u}\}
    \end{align*}
    by applying Theorem~\ref{thm:sublevel_1} (using A1-A6) to the case where $\mu_0 = \delta_{x_0}$. We remark that our $v_k$ is feasible for the dual problem on functions \eqref{eq:risk_cont}, so  by optimality of $p^*(t_0, x_0)$ for each $x_0 \in X_0$, one has $v_k(t_0,x_0) \geq p^*(t_0,x_0)$. Moreover, as $p^*(t_0, \mu_0) \leq 1 < \infty$ by \rw{Proposition  \ref{prop:upper_bound}}, each
    $v_k(t_0,\bullet)$ is $\mu_0$-integrable, and thus the function $v^\star(t_0,\bullet)$ is $\mu_0$-integrable as well. It remains to prove that their difference $v^k(t_0, \bullet) - p^*(t_0, \bullet)$ ultimately vanishes in a $\mu_0$-sense. This vanishing difference is proven  by contradiction: suppose that there exists a $\eta > 0$ such that for all $k \in \N$, there is a $\varphi(k) \geq k$ with $\inp{v_{\varphi(k)}(t_0,\bullet) - p^*(t_0,\bullet)}{\mu_0} \geq \eta$ (which is the negation of \eqref{eq:l1_converge}). Then, one has
    $$\inp{p^*(t_0,\bullet)}{\mu_0} \leq \inp{v_{\varphi(k)}(t_0,\bullet)}{\mu_0} - \eta \underset{k\to\infty}{\longrightarrow} J^*(t_0, \mu_0) - \eta.$$
    However, as we already noticed, $p^*(t_0, x_0)$ is feasible for~\eqref{eq:risk_cont}, so this is a contradiction with optimality of the value $J^*(t_0, \mu_0)$ (given that $J^* - \eta < J^*$). Such a contradiction proves \eqref{eq:l1_converge}, and therefore proves Theorem \ref{thm:converge_sequence_init}.

%% file: appendix/app_converge_sos.tex
\section{Convergence of the Moment-SOS hierarchy}

\label{app:moment_sos_converge}
\rw{
This appendix proves convergence of the \ac{SOS}-hierarchy to the unsafety \acp{LP} (Theorem \ref{thm:sos_converge}).
To begin, we require a lemma that establishes boundedness.}

\begin{lem}
\label{lem:bounded_meas}
    Under Assumptions A1-A6, all feasible measures $\mu_0, \mu, \mu_c, \mu_{u}$ in \eqref{eq:unsafe_meas} are bounded.
\end{lem}
\begin{proof}
    Boundedness of nonnegative measures will be proven by the sufficient condition of compact support and finite mass. Compact support of measures in $\bbmu$ is ensured by A1. The initial measure $\mu_0$ is a probability distribution $(\inp{1}{\mu_0} = 1)$ by constraint \eqref{eq:unsafe_meas_prob}. The sum $\mu_c + \mu_{\rw{u}}$ likewise has mass 1 by the Liouville constraint \eqref{eq:unsafe_meas_liou}, when passing in the test function $v = 1$. Given that $\mu_c$ and $\mu_{\rw{u}}$ are each nonnegative measures, it holds that they both have finite mass (upper-bounded by 1). Lastly, assignment of $v = t$ to \eqref{eq:unsafe_meas_liou} with $\Lie t = 1$ results in $\inp{1}{\mu} = \inp{t}{\mu_c + \mu_{\rw{u}}} \leq T$. All measures have bounded masses and compact supports, and therefore are bounded.
\end{proof}
\rw{
\begin{cor}
\label{cor:bounded_risk}
    Under A1-A6, all measures $\mu, \mu_c, \mu_{\rw{u}}$ that are feasible for \eqref{eq:unsafe_meas_given} are bounded if the given  $\mu_0$ is a probability distribution.
\end{cor}
\begin{proof}
    This immediately follows from Lemma \ref{lem:bounded_meas}, because the tuple $(\mu_0, \mu_c, \mu_{\rw{u}})$ are bounded is a bounded solution to \eqref{eq:unsafe_meas_liou}-\eqref{eq:unsafe_meas_support}. 
\end{proof}
}
\rw{It is now possible to prove Theorem \ref{thm:converge_sequence_init} by employing 
Corollary 8 of \cite{tacchi2022convergence} because:
\begin{enumerate}    
    \item   All sets are \ac{BSA} and Archimedean (A7).
    \item All dynamics are polynomial (A8).
    \item All measures are bounded (Lemma \ref{lem:bounded_meas} and Cor. \ref{cor:bounded_risk}).
    \item The objective value is finite ($d^*_k(t_0, X_0) > 0$). $\qed$
    \end{enumerate}}